\documentclass[a4paper]{article}
\usepackage{latexsym}
\usepackage{amsmath,amssymb,amsfonts,graphicx}

\newtheorem{prethm}{{\bf Theorem}}

\newenvironment{thm}{\begin{prethm}{\hspace{-0.5
               em}{\bf.}}}{\end{prethm}}

\newtheorem{prepro}[prethm]{Proposition}

\newtheorem{prelem}[prethm]{Lemma}
\newenvironment{lem}{\begin{prelem}{\hspace{-0.5
               em}{\bf.}}}{\end{prelem}}

\newtheorem{precor}[prethm]{Corollary}

\newtheorem{prerem}[prethm]{{\bf Remark}}
\newenvironment{rem}{\begin{prerem}\em{\hspace{-0.5
              em}{\bf.}}}{\end{prerem}}

\newtheorem{preconj}[prethm]{{\bf Conjecture}}

\newtheorem{preexample}{{\bf Example}}

\newtheorem{preproof}{{\bf Proof.}}

\newenvironment{proof}[1]{\begin{preproof}{\rm
               #1}\hfill{$\Box$}}{\end{preproof}}

\newcommand{\al}{\alpha}
\newcommand{\la}{\lambda}

\newcommand{\be}{\beta}
\newcommand{\e}{{\rm E}}
\newcommand{\he}{{\rm HE}}
\renewcommand{\thefootnote}

\begin{document}

\title{Bounds for the H\"uckel energy of a graph}

 \author{Ebrahim Ghorbani$^{\,\rm a,b,}$\footnote{
This work was done  while the first
author was visiting the department of mathematics of POSTECH. He would like
to thank the department for its hospitality and support.}
,~Jack H. Koolen$^{\,\rm c,d,}$\footnote{
He was partially supported by
a grant from the Korea Research Foundation funded by the Korean government (MOEHRD)
under grant number KRF-2008-314-C00007.}
,~  Jae Young Yang$^{\,\rm c}$\\
{\footnotesize {$^{\rm a}$Department of Mathematical Sciences, Sharif University of Technology,}}\\
{\footnotesize {P.O. Box 11155-9415,
 Tehran, Iran}}\\
{\footnotesize {$^{\rm b}$School of Mathematics, Institute for Research in Fundamental Sciences (IPM),}}\\
 {\footnotesize {P.O. Box 19395-5746, Tehran, Iran}}\\
{\footnotesize{$^{\rm c}$Department of Mathematics,  POSTECH, Pohang 790-785, South Korea}}\\
{\footnotesize{$^{\rm d}$Pohang Mathematics Institute,  POSTECH, Pohang 790-785, South Korea}}}

\footnotetext{E-mail Addresses: {\tt e\_ghorbani@math.sharif.edu} (E. Ghorbani), {\tt koolen@postech.ac.kr} (J.H. Koolen),
{\tt rafle@postech.ac.kr} (J.Y. Yang)}

\date{\today}

\maketitle

\begin{abstract}
Let $G$ be a  graph on $n$ vertices with $r := \lfloor n/2 \rfloor$ and let $\lambda _1
\geq\cdots\geq \lambda _{n} $ be adjacency eigenvalues of $G$.  Then the H\"uckel energy of $G$,  HE($G$), is
defined as
$$\he(G) =\left\{
  \begin{array}{ll}
    2\sum_{i=1}^{r} \lambda_i, & \hbox{if $n= 2r$;} \\
    2\sum_{i=1}^{r} \lambda_i + \lambda_{r+1}, & \hbox{if $n= 2r+1$.}
  \end{array}
\right.$$
The concept of H\"uckel energy was introduced by  Coulson as it gives a good approximation for the $\pi$-electron energy of molecular graphs.
We obtain two upper bounds and a lower bound for HE$(G)$. When $n$ is even, it is  shown that equality holds in both upper bounds if and only if $G$ is a strongly regular graph with parameters $(n, k, \lambda, \mu) = (4t^2 +4t +2,\, 2t^2 +3t +1,\,
t^2 +2t,\, t^2 + 2t +1),$ for positive integer $t$.
Furthermore, we will give an infinite family of these strongly regular graph whose construction was communicated by Willem Haemers to us. He attributes the construction to J.J. Seidel.

\end{abstract}

\section{Introduction}

\noindent
Throughout  this paper,  all graphs are simple and undirected.
Let $G$ be a graph with $n$ vertices and $A$ be the  adjacency
matrix of $G$. Then the eigenvalues of  $G$ are defined as the eigenvalues of
$A$. As all eigenvalues of $A$ are real, we can rearrange them as
$\lambda _1 \geq\cdots\geq \lambda _{n} $.
I. Gutman (see \cite{Gut}) defined the energy of $G$, E$(G)$, by
$$\e(G) := \sum_{i=1}^{n} | \lambda_i |.$$
In chemistry, the energy of a given molecular graph is of interest since it
can be related to the total $\pi$-electron energy of the molecule represented by that graph.
The reason for Gutman's definition is that E$(G)$ gives a good approximation for the $\pi$-electron energy of a molecule where $G$ is then the corresponding molecular graph. For a survey on the energy of graphs, see \cite{Gut}.
 The  {\em H\"uckel energy} of $G$, denoted by HE$(G)$, is defined as
$$\he(G) =\left\{
  \begin{array}{ll}
    2\sum_{i=1}^{r} \lambda_i, & \hbox{if $n= 2r$;} \\
    2\sum_{i=1}^{r} \lambda_i + \lambda_{r+1}, & \hbox{if $n= 2r+1$.}
  \end{array}
\right.$$
The idea of introducing H\"uckel energy (implicitly) exists in Erich H\"uckel's first
paper \cite{h31} in 1931 and is also found  in his book \cite{h40}.
The concept was explicitly used in 1940 by Charles Coulson~\cite{c}
but, most probably, can be found also in his earlier articles. In a ``canonical"
form, the theory behind the H\"uckel energy was formulated in a series of
papers by Coulson and Longuet-Higgins, of which the first (and most important)
is \cite{cl}. In comparison with energy, the  H\"uckel energy of a graph gives a better approximation for the total $\pi$-electron energy of the molecule represented by that graph, see \cite{Fow}.
Clearly for a graph on $n $ vertices, HE$(G) \leq$ E$(G)$, and if $G$ is bipartite, then equality holds. Koolen and Moulton in \cite{KoMo01, KoMo03} gave upper bounds on
the energy of graphs and bipartite graphs, respectively. These bounds have been generalized in several ways. Obviously, the upper bounds of Koolen and Moulton also give  upper bounds for the H\"uckel energy of graphs. In this paper, we obtain better upper bounds for the H\"uckel energy of a graph. More precisely, we prove that for a graph $G$ with $n$ vertices and $m$ edges,
\begin{equation}\label{even}
\he(G) \leq\left\{
  \begin{array}{ll}
    \frac{2m}{n-1}+ \frac{\sqrt{2m(n-2)\left(n^2-n-2m\right)}}{n-1}, & \hbox{if $m\le\frac{n^3}{2(n+2)}$;} \\
    \frac{2}{n}\sqrt{mn(n^2-2m)}<\frac{4m}{n}, & \hbox{otherwise;}
  \end{array}
\right.
\end{equation}
if $n$ is even, and
\begin{equation}\label{odd}
  \he(G)\le\left\{
  \begin{array}{ll}
    \frac{2m}{n-1}+\frac{\sqrt{2mn(n^2-3n+1)(n^2-n-2m)}}{n(n-1)}
& \hbox{if $m\le\frac{n^2(n-3)^2}{2(n^2-4n+11)}$;} \\
    \frac{1}{n}{\small\sqrt{2m(2n-1)(n^2-2m)}}, & \hbox{otherwise;}
  \end{array}
\right.
\end{equation}
 if $n$ is odd. Then we show that
\begin{equation}\label{nbnd}
   \he(G)\le \frac{n}{2}\left(1+\sqrt{n-1}\right),
\end{equation}
if $n$ is even, and
\begin{equation}\label{nbodd}
\he(G)< \frac{n}{2}\left(1+\sqrt{n}-\frac{1}{\sqrt{n}}\right),
\end{equation}
if $n$ is odd. Moreover, equality is attained  in (\ref{even})  if and only if equality  attained  in (\ref{nbnd}) if and only if $G$ is a strongly regular graph with  parameters $(n, k, \lambda, \mu) =
(4t^2 +4t +2,\, 2t^2 +3t +1,\,
t^2 +2t,\, t^2 + 2t +1)$.
The proofs of the above upper bounds are given in Section~2.
It is known that $\e(G)\ge 2\sqrt{n-1}$ for any graph $G$ on $n$ vertices with no isolated vertices with equality if and only if
$G$ is the star $K_{1,n-1}$ (see \cite{ccgh}). In Section~3, we prove that the same bound holds for H\"uckel energy.
In the last section, we give a construction of srg$(4t^2 +4t +2,\, 2t^2 +3t +1,\,
t^2 +2t,\, t^2 + 2t +1)$.


\section{The upper bound for H\"uckel  energy}

In this section we prove (\ref{even}),  (\ref{odd}), (\ref{nbnd}), and (\ref{nbodd}). The equality cases are also discussed.
We begin by stating a lemma which will be used later.

\begin{lem}\label{lem} Let  $G$ be a graph with $n$ vertices and $m$ edges. Suppose  $r:=\lfloor n/2\rfloor$ and $$\al:=\sum_{i=1}^r\la_i^2(G).$$
 If $m\ge n-1\ge2$, then
\begin{equation}\label{almn}
    \frac{\al}{r}\le\frac{4m^2}{n^2}.
\end{equation}
\end{lem}
\begin{proof}{First, suppose $m\ge n$. Then $G$ contains a cycle, and so by interlacing, we see
$\la_n^2+\la_{n-1}^2\ge2.$
Therefore, $\al/r\le(2m-2)/r\le4m^2/n^2. $
 If $m=n-1$ and $G$ is connected, then $G$ is a tree. Thus  $\al= m$, and   obviously (\ref{almn}) holds.
So in the rest of proof we assume that $G$ is disconnected and $m=n-1$.
If $G$ has at least three two non-trivial components, then at least one of the components contains a cycle.
The component containing a cycle has either an eigenvalue at most $\le-2$, or two eigenvalues $\le-1$ and the other two components
have an eigenvalue $\le-1$.
It turns out that
$\la_{n}^2+\la_{n-1}^2+\la_{n-2}^2+\la_{n-3}^2\ge4$ where $n\ge7$.
Hence
\begin{equation}\label{almn2}
 \al\le2m-4.
\end{equation}
It is easily seen that (\ref{almn}) follows from (\ref{almn2}).
Now, suppose that $G$ has two non-trivial connected components $G_1$ and $G_2$, say. Let $G_1, G_2$ have $n_1,n_2$ vertices and $m_1,m_2$ edges, respectively. First suppose $n_1,n_2\ge3$. If $G_1$ or $G_2$ contains a $K_{1,2}$ as an induced subgraph, we are done by interlacing.  So one may assume that both $G_1$ and $G_2$ contain a triangle.  It turns out that $m_1\ge n_1$ and $m_2\ge n_2$. Hence $G$ must have an isolated vertex which implies $n\ge7$. On the other hand, by interlacing, the four smallest eigenvalues of $G$ are at most $-1$ implying (\ref{almn2}). Now, assume that $n_1=2$. So $G_2$ must contain a cycle $C_{\ell}$. We may assume that $C_{\ell}$ has no chord. If $\ell\ge4$, we are done. So let $\ell=3$. If $n_2=3$, then $G$  does not have isolated vertices, i.e., $G=C_3\cup K_2$, for which (\ref{almn}) holds. Thus $n_2\ge4$ which means that at least one of the diamond graph, the paw graph (see Figure~\ref{fig}), or $K_4$ is induced subgraph of $G$. If it contains either the diamond or the paw graph, we are done by interlacing. If it contains $K_4$, then $G$ must have at least two isolated vertices, i.e., $n\ge8$. Thus the four smallest eigenvalues of $G$ are at most $-1$ which implies (\ref{almn2}).
Finally, assume that $G$ has exactly one non-trivial component $G_1$ with $n_1$ vertices. It turns out that $n_1\ge3$ and $G_1$ contains a cycle.
By looking at the table of graph spectra of \cite[pp. 272--3]{cds}, it is seen that if $n_1=3,4$, $G$ satisfies (\ref{almn}).
If $n_1\ge5$, then making use of the table mentioned above  and interlacing it follows that $\la_{n}^2+\la_{n-1}^2+\la_{n-2}^2\ge4$
unless $G_1$ is a complete graph in which case $G$ has at least 11 vertices and the four smallest eigenvalues of $G$ is $-1$, implying the result.}
\end{proof}

\begin{figure}
\begin{center}  
  \includegraphics[width=2cm]{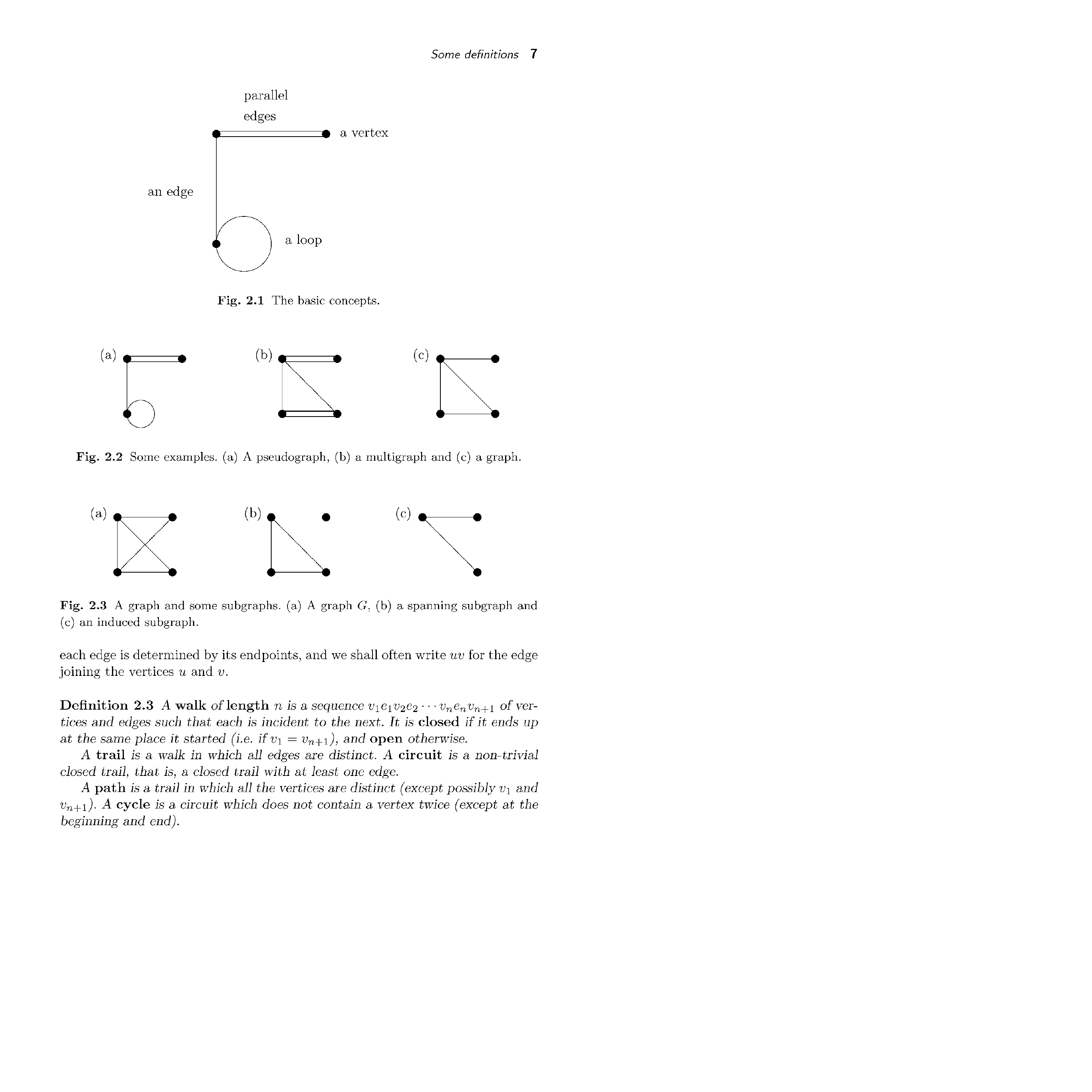}
\hspace{1.5cm}
\includegraphics[width=2cm]{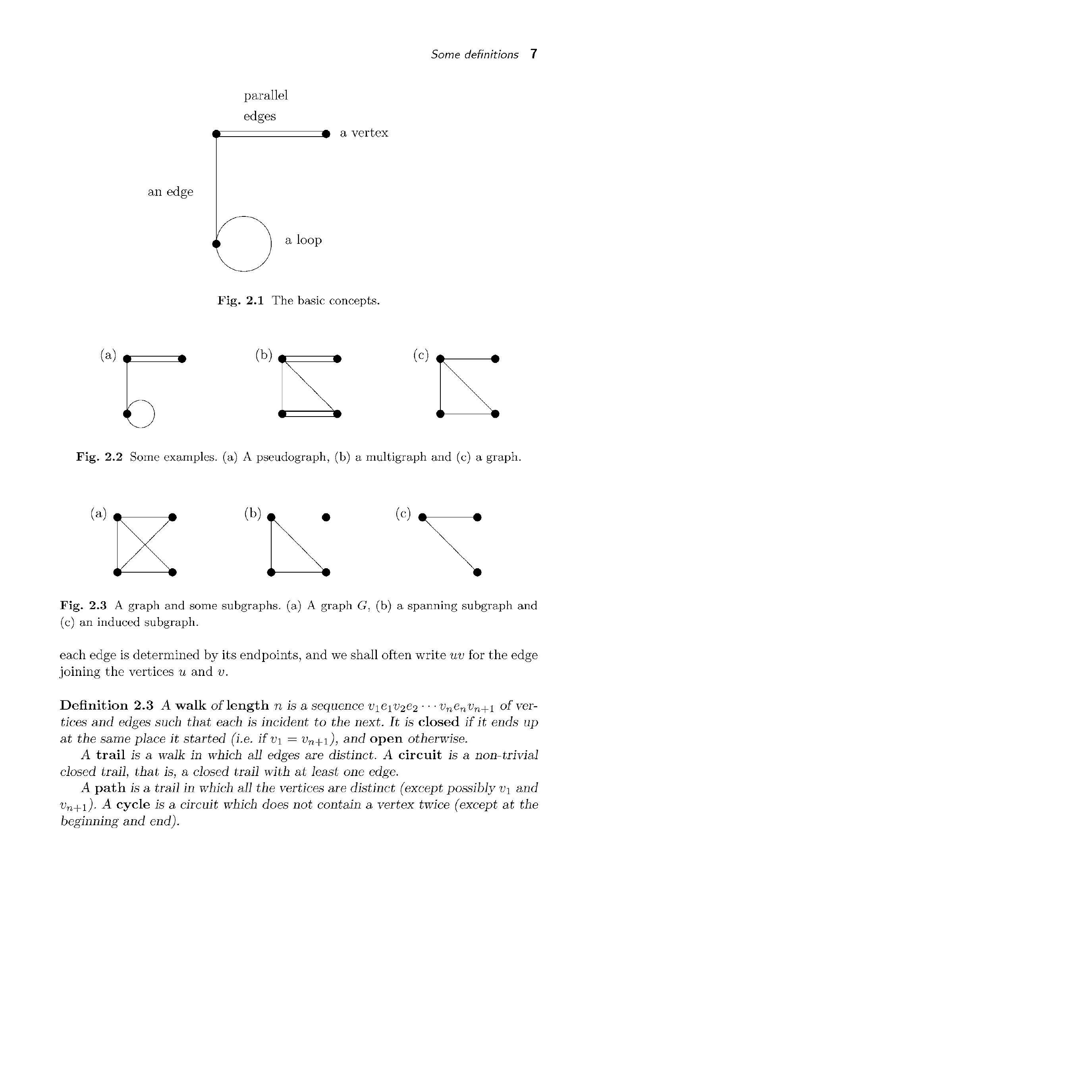}
  \caption{The paw and the diamond graphs}\label{fig}
\end{center}
\end{figure}

\subsection{Even order graphs}

In this subsection we prove (\ref{even}) and (\ref{nbnd}) for graphs of an even order. The cases of equalities are also characterized.

\begin{thm}\label{thmeven}
Let $G$ be a graph on $n$ vertices and $m$ edges where $n$ is even. Then (\ref{even}) holds.
Moreover, equality holds if and only if $n = 4t^2 + 4t +2$ for some positive integer $t$, $m = (2t^2+2t+1)(2t^2+3t+1)$ and $G$ is a strongly regular graph with  parameters $(n, k, \lambda, \mu) =
(4t^2 +4t +2,\, 2t^2 +3t +1,\,
t^2 +2t,\, t^2 + 2t +1)$.
\end{thm}
\begin{proof}{Let $n=2r$ and $\lambda_1 \geq \lambda_2 \geq \cdots \geq \lambda_n$ be the eigenvalues of $G$. Then
$$\sum_{i=1}^n \lambda_i = 0 \ \ \ \hbox{and } \ \ \
\sum_{i=1}^n \lambda_i^2 = 2m.$$
Let $\al$ be as in Lemma~\ref{lem}.
Then
$$2m- \alpha = \sum_{i=r+1}^{n} \lambda_i^2,$$ and
$\lambda_1^2 \leq \alpha \le 2m-2$.
 By the Cauchy-Schwartz inequality,
$$\he(G) = 2\sum_{i=1}^{r} \lambda_i \leq
2 \lambda_1 + 2\sqrt{(r-1)(\alpha - \lambda_1^2)}.$$
The function $x\mapsto x+\sqrt{(r-1)(\alpha - x^2)}$  decreases on the interval  $\sqrt{\al/r}\le x\le\sqrt\al$. By Lemma~\ref{lem}, 
 $m/r\ge\sqrt{\al/r}$. Since
 $\lambda_1 \ge m/r$, we see that
\begin{equation}\label{h1}
  \he(G)\leq f_1(\al):= \frac{2m}{r} + 2\sqrt{(r-1)(\alpha- m^2/r^2)}.
\end{equation}
On the other hand,
\begin{equation}\label{h2}
  \he(G) = - 2\sum_{i=r+1}^{n} \lambda_i \leq
 f_2(\al):=2\sqrt{r(2m-\alpha)}.
\end{equation}
Let $$f(\al):=\min\{f_1(\al),f_2(\al)\}.$$
We determine the maximum of $f$. We observe that
 $f_1$ and $f_2$ are increasing and  decreasing functions in $\alpha$, respectively.
Therefore, $\max\,f=f(\al_0)$ where $\al_0$ is the unique point with $f_1(\al_0)=f_2(\al_0)$.
So we find the solution of the equation $f_1(\al)=f_2(\al)$. To do so,
let $\sigma:= \sqrt{\alpha - m^2/r^2}$ and consider the equation
$$\frac{m}{r} +\sigma \sqrt{r-1}= \sqrt{r(2m-\sigma^2 - m^2/r^2)}.$$
This equation has the roots
$$\sigma_{1,2}:=\frac{-m\sqrt{r-1} \pm r\sqrt{2m(2r^2-r-m)}}{r(2r-1)}.$$
If $m \le2r^3/(r+1)$, then  $\sigma_1\ge0$ and so
\begin{align}\label{even1}
\he(G)& \leq \frac{2m}{r}+ \frac{2\sqrt{r-1}}{r(2r-1)}\left(-m\sqrt{r-1}+r\sqrt{2m(2r^2-r-m)}\right)\nonumber\\
&=   \frac{2m}{n-1}+ \frac{\sqrt{2m(n-2)\left(n^2-n-2m\right)}}{n-1};
\end{align}
otherwise
\begin{equation}\label{even2}
\he(G) \leq2\sqrt{r(2m-m^2/r^2)},
\end{equation}
which is less than $4m/n$ for $m>2r^3/(r+1)$.
This shows that Inequality~(\ref{even}) holds.

Now let us consider the case that equality is attained in  (\ref{even}).
First let $m \leq \frac{2r^3}{r+1}$. Then equality holds if and only if \begin{enumerate}
\item $\lambda_1 = \frac{m}{r}$;
\item $\lambda_2 = \cdots = \lambda_{r} = \frac{\sigma_1}{\sqrt{r-1}}$;
\item $\lambda_{r+1}= \cdots = \lambda_n = -\frac{1}{\sqrt{r}}\sqrt{2m-\sigma_1^2-m^2/r^2}$.
\end{enumerate}
The first condition shows that $G$ must be $\frac{m}{r}$-regular, and the second and third conditions imply that $G$ must be strongly regular graph as a regular graph with at most three distinct eigenvalues is strongly regular, cf. \cite[Lemma 10.2.1]{Godsil}. From \cite[Lemma 10.3.5]{Godsil} it follows that $G$ has to have the parameters as required in the theorem.
If $m > \frac{2r^3}{r+1}$, then with the same reasoning as above one can show that $G$ has to be strongly regular graph with eigenvalue 0 of multiplicity $r-1$ and by   \cite[Lemma 10.3.5]{Godsil} such a graph does not exist.}
\end{proof}

Optimizing the H\"uckel energy over the number of edges we obtain:
\begin{thm} Let $G$ be a graph on $n$ vertices where $n$ is even. Then (\ref{nbnd}) holds.
Equality holds if and only if  $G$ is a strongly regular graph with  parameters
$(4t^2 +4t +2,\, 2t^2 +3t +1,\,
t^2 +2t,\, t^2 + 2t +1)$, for some positive integer $t$.
\end{thm}
\begin{proof}{Suppose that $G$ is a graph with $n$ vertices and $m$ edges. If $m\le n-2$, then (\ref{nbnd}) obviously holds as $\e(G)\le2m$ (see \cite{Gut}).  If $m\ge n-1$, then using routine calculus, it is seen that the right hand side
of (\ref{even1})|considered as a function of $m$|is maximized when
$$m=n(n-1+\sqrt{n-1})/4.$$ Inequality~(\ref{nbnd}) now follows by substituting this value of $m$ into (\ref{even}).
(We note that the maximum of the right hand side of (\ref{even2}) is $2n^3/(n+2)$ which is less than the above maximum.)
Moreover,  from Theorem~\ref{thmeven} it follows that equality holds in (\ref{nbnd})  if
and only if $G$ is a strongly regular graph with parameters $(4t^2 +4t +2,\, 2t^2 +3t +1,\,
t^2 +2t,\, t^2 + 2t +1)$.}
\end{proof}

\subsection{Odd order graphs}

In this subsection we prove (\ref{even}) and (\ref{nbodd}) for graphs of an odd order and discuss the equality case and tightness of the bounds.

\begin{thm}\label{thmodd} Let $m\ge n-1\ge3$ and $G$ be a graph with  $n$ vertices and $m$ edges where $n$ is odd. Then (\ref{odd}) holds.
\end{thm}

\begin{proof}{Let $n=2r+1$, $\al$ be as before, and
$$\be:=\la_{r+1}.$$
We have
\begin{equation*}
    2m-\al\ge(r+1)\be^2.
\end{equation*}
This obviously holds if $\be\le0$. For $\be\ge0$ it follows from the following:
\begin{align*}
    2m-\al-\be^2&=\sum_{i=r+2}^n\la_i^2
\ge\frac{1}{r}\left(\sum_{i=r+2}^n\la_i\right)^2\\
&=\frac{1}{r}\left(\sum_{i=1}^{r+1}\la_i\right)^2
\ge\frac{1}{r}(r+1)^2\be^2,
\end{align*}
where the first inequality follows from the Cauchy-Schwartz inequality.
In a similar manner as the proof of Theorem~\ref{thmeven}, we find that $\he(G)\le\min\{f_1(\al,\be),f_2(\al,\be)\}$, where
\begin{align}
   f_1(\al,\be)&=4m/n+2\sqrt{(r-1)\left(\al-4m^2/n^2\right)}+\be,~~\hbox{and}\label{f1}\\
   f_2(\al,\be)&=2\sqrt{r(2m-\al-\be^2)}-\be.\label{f2}
\end{align}
Let $$f(\al,\be):=\min\left\{f_1(\al,\be),f_2(\al,\be)\right\}.$$ We determine the maximum of $f$ over the compact set
$$D:=\left\{(\al,\be): \al\ge4m^2/n^2,\,2m-(r+1)\be^2\ge\al \right\}.$$
Note that for $(\al,\be)\in D$ one has $-\be_0\le\be\le\be_0$, where
$$\be_0=\frac{2}{n}\sqrt{\frac{m(n^2-2m)}{n+1}}.$$
Neither the gradient of $f_1$ nor that of $f_2$ has a zero in interior of $D$.
So the maximum of $f$ occurs  in the set  $$L:=\{(\al,\be):f_1(\al,\be)=f_2(\al,\be)\},$$
 where the gradient of $f$ does not exist or it occurs
in the boundary of $D$ consisting of
\begin{align*}
   D_1&:=\{(\al,\be): \al=4m^2/n^2,\,-\be_0\le\be\le\be_0 \},\\
   D_2&:=\{(\al,\be): \al=2m-(r+1)\be^2,\,-\be_0\le\be\le\be_0 \}.
\end{align*}

For any $(\al,\be)\in D$,  $f_2(\al,\be)\le f_2(4m^2/n^2,\be)$. It is easily seen that the maximum of $f_2(4m^2/n^2,\be)$ occurs in
$$\be_1:=-\frac{1}{n}\sqrt{\frac{2m(n^2-2m)}{2n-1}}.$$
Therefore,
\begin{equation}\label{maxf2}
  \max\,f_2=f_2(4m^2/n^2,\be_1)=\frac{1}{n}\sqrt{2m(2n-1)(n^2-2m)}.
\end{equation}
In the rest of proof, we determine $\max\,f$ for
\begin{equation}\label{m}
  m\le\frac{n^2(n-3)^2}{2(n^2-4n+11)};
\end{equation}
 if $m>\frac{n^2(n-3)^2}{2(n^2-4n+11)}$, then (\ref{odd}) follows from (\ref{maxf2}).

On $D_1$, we have
$$\max\,f_{|_{D_1}}\le f_1(4m^2/n^2,\be_0)=4m/n+\be_0.$$

On $D_2$, one has
\begin{align*}
   f_1(\be)&=4m/n+2\sqrt{(r-1)(2m-(r+1)\be^2-4m^2/n^2)}+\be,~~\hbox{and}\\
   f_2(\be)&=(n-1)|\be|-\be.
\end{align*}
In order to find  $\max\,f_{|_{D_2}}$, we look for the points where $f_1(\be)=f_2(\be)$. The solution of this equation for $\be\le0$ is
$$\be_{2}=\frac{-2m(n+1)-\sqrt{2mn(n^2-2n-3)(n^2-n-2m)}}{n(n^2-1)},$$
and for $\be\ge0$ is
$$\be_{3}=\frac{2m(n-3)+\sqrt{2m(n-3)(n^4-4n^3-2mn^2+3n^2+6mn-8m)}}{n(n^2-4n+3)}.$$
We have $\be_2\ge-\be_0$  if and only if $m\le\frac{n^2(n+1)}{2(n+3)}$, and $\be_3\le\be_0$ if and only if $m\le\frac{n^2(n-3)^2}{2(n^2-4n+11)}$.
Moreover $f_2(\be_2)>f_2(\be_3)$. Thus if $m\le\frac{n^2(n-3)^2}{2(n^2-4n+11)}$,
 $$\max\,f_{|_{D_2}}=f_2(\be_2)=\frac{2m(n+1)+\sqrt{2mn(n^2-2n-3)(n^2-n-2m)}}{n^2-1}.$$

Now we examine $\max f_{|_L}$.
Let $\sigma:=\sqrt{\al-4m^2/n^2}$.  To determine $(\al,\be)$ satisfying $f_1(\al,\be)=f_2(\al,\be)$ it is enough to find the zeros of the following quadratic form:
\begin{equation}\label{sig}
(2n-4)\sigma^2+4(2m/n+\be)\sqrt{n-3}\sigma+(4m/n+2\be)^2-(n-1)(4m-2\be^2-8m^2/n^2).
\end{equation}
The zeros are
$$\sigma_{1,2}:=\frac{-(2m+n\be)\sqrt{n-3}\pm\sqrt{(n-1)(2mn^3-\be^2n^3+\be^2n^2-4mn^2-4m\be n-4nm^2+4m^2)}}{n(n-2)}.$$
Note that  $\sigma_2<0$ and so is not feasible. Let us denote the constant term of (\ref{sig}) by $h(\be)$ as a function of $-\be_0\le\be\le\be_0$. Then $\sigma_1\ge0$ if and only if  $h(\be)\le0$. Moreover $h(\be)\le h(\be_0)$, 
and $h(\be_0)\le0$ if
$$m\le\frac{n^2(n-3)^2}{2(n^2-4n+11)}.$$
 Thus, with this condition on $m$, $f_1$ becomes $f_1(\be)=4m/n+2\sqrt{r-1}\sigma_1$. The roots of $f_1'(\be)=0$ are
$$\be_{4,5}=\frac{-2m(n^2-3n+1)\pm\sqrt{2mn(n^2-3n+1)(n^2-n-2m)}}{n(n^3-4n^2+4n-1)}.$$
It is seen that $-\be_0\le\be_5\le\be_4\le0$ unless $m\le n^2/(2n-2)$. We have
$$f_1(\be_4)-f_1(\be_5)=\frac{2\sqrt{2mn(n^2-3n+1)(n^2-n-2m)}}{n(n-2)(n^3-4n^2+4n-1)}.$$
Thus $f_1(\be_4)\ge f_1(\be_5)$. It turns out that $f_1$ decreases for $\be\ge\be_4$, so $f_1(\be_4)\ge f_1(\be_0)$.
 It is easily seen that $f_1(\be_4)\ge f_1(-\be_0)$.
Therefore, for $m\le\frac{n^2(n-3)^2}{2(n^2-4n+11)}$ we have
\begin{equation}\label{maxL}
    \max\,f_{|_{L}}=
                 f_1(\be_4)=\frac{2m}{n-1}+\frac{\sqrt{2mn(n^2-3n+1)(n^2-n-2m)}}{n(n-1)}.
\end{equation}
The result follows from comparing the three maxima $\max\,f_{|_{D_1}}$, $\max\,f_{|_{D_2}}$, and $\max\,f_{|_{L}}$.}
\end{proof}

\begin{thm} Let $G$ be a graph on $n$ vertices where $n$ is odd. Then (\ref{nbodd}) holds.
\end{thm}
\begin{proof}{The maximum of the right hand side of (\ref{maxf2})|as a function of $m$| is
\begin{equation}\label{mf2m}
 \frac{n(n-3)\sqrt{2(n+1)(2n-1)}}{n^2-4n+11},
\end{equation}
which is obtained when $m$ is given by (\ref{m}). Also, the right hand side of (\ref{maxL}) is maximized when
$$m=\frac{n(n-1+\sqrt{n})}{4}.$$
This maximum value is equal to $\frac{n}{2}\left(1+\sqrt{n}-\frac{1}{\sqrt{n}}\right)$
 which  is greater than (\ref{mf2m}). This completes the proof.}
\end{proof}

\begin{rem}
 Here we show that no graph can attain the bound in (\ref{odd}). Let us keep the notation of the proof of Theorem~\ref{thmodd}.
First we consider $m>\frac{1}{2}n^2(n-3)^2/(n^2-4n+11).$
Therefore,  $\he(G)$ equals (\ref{maxf2}). Then $\al=4m^2/n^2$ and $\la_{r+1}=\be_1$. This means that $G$ is a regular graph with only one positive eigenvalue.
Then by \cite[Theorem 6.7]{cds}, $G$ is a complete multipartite graph. As the rank of a complete multipartite graph equals the number of its parts, $G$ must have $r+2$ parts. Such a graph cannot be regular, a contradiction.
Now, we consider $m\le\frac{1}{2}n^2(n-3)^2/(n^2-4n+11).$ Hence $\he(G)$ is equal to (\ref{maxL}). Then $G$ must be a regular graph of degree $k$, say, with $\la_2=\cdots=\la_r$, $\la_{r+1}=\be_4$, and $\la_{r+2}=\cdots=\la_n$.
 Since $\la_{r+1}=\be_4<0$, $\la_2$ and $\la_n$ have different multiplicities, and thus all eigenvalues of $G$ are integral. Let $\la_2=t$. Then $\la_n=-t-s$, for some integer $s\ge0$. It follows that $k+\la_{r+1}=t+rs$. This implies that either $s=0$ or $s=1$. If $s=0$, then $k+\la_{r+1}=t$, and so $\he(G)=k+(n-2)t$. This must be equal to (\ref{maxL}) which implies
$$t=\frac{k+\sqrt{nk(n^2-3n+1)(n-1-k)}}{n^2-3n+2}.$$
Substituting this value of $t$ in the equation $nk=k^2+(n-2)t^2+(t-k)^2$ and solving in terms of $k$ yields $k=n/(n-1)$ which is impossible. If $s=1$, then $k+\la_{r+1}=t+r$, and so $\he(G)=k+(n-2)t+r$. It follows that
$$t=\frac{k-(n-1)^2/2+\sqrt{nk(n^2-3n+1)(n-1-k)}}{n^2-3n+2}.$$
Substituting this value of $t$ in the equation $nk=k^2+(r-1)t^2+r(t+1)^2+(r+t-k)^2$ and solving in terms of $k$ yields $k=(n-1+\sqrt n)/2$  which implies $t=(\sqrt{n}-1)/2$. Therefore, we have $\la_{r+1}=-\frac{1}{2}$, a contradiction.
\end{rem}

\begin{rem} 
Note that a conference strongly regular graph $G$, i.e, a srg$(4t+1,\,2t,\,t-1,\,t)$, has spectrum $\left\{[2t]^1, [(-1+\sqrt{4t+1})/2]^{2t}, [(-1-\sqrt{4t+1})/2]^{2t} \right\}$, and hence $\he(G)=\frac{2t+1}{2}(1+\sqrt{4t+1})$. This is about half of the upper bound given in (\ref{nbodd}). For odd order graphs, we can come much closer to (\ref{nbodd}).
Let $G$ be a ${\rm srg}(4t^2 +4t +2,\, 2t^2 +3t +1,\,t^2 +2t,\, t^2 + 2t +1)$.
 If one adds a new vertex to $G$ and join it to  neighbors of some fixed vertex of $G$, then the resulting graph $H$ has the spectrum
$$\left\{[\la_1]^1,\,[t]^{2t^2+2t-1},\,[\la_2]^1,\,[0]^1,\,[-t-1]^{2t^2+2t},\,[\la_3]^1\right\},$$
where $\la_1,\la_2,\la_3$ are the roots of the polynomial
$$p(x):=x^3-(2t^2+3t)x^2-(5t^2+7t+2)x+4t^4+10t^3+8t^2+2t.$$
It turns out that $p(-\sqrt2t)=(4+3\sqrt2)t^3+(8+7\sqrt2)t^2+(2+2\sqrt2)t$. Hence $\la_3<-\sqrt2t$ and so
$$\he(H)>2(2t^2+2t)(t+1)+2\sqrt2t=\frac{4t^2+4t+3}{2}(\sqrt{4t^2+4t+3})+O(4t^2+4t+3).$$
This shows that (\ref{nbodd}) is asymptotically tight.
\end{rem}

\section{Lower bound}

It is known that $\e(G)\ge 2\sqrt{n-1}$ for any graph $G$ on $n$ vertices with no isolated vertices with equality if and only if
$G$ is the star $K_{1,n-1}$ (see \cite{ccgh}). Below we show that this is also the case for H\"uckel energy.

\begin{thm} For any graph $G$ on $n$ vertices with no isolated vertices, $$\he(G)\ge 2\sqrt{n-1}.$$ Equality holds if and only if $G$ is the star $K_{1,n-1}$.
\end{thm}
\begin{proof}{ If $G_1,G_2$ are two graphs with $n_1,n_2$ vertices, then
$\he(G_1\cup G_2)\ge\he(G_1)+\he(G_2)$, and $\sqrt{n_1-1}+\sqrt{n_1-1}\ge\sqrt{n_1+n_2-1}$ for $n_1,n_2\ge2$.
This alows us to assume that $G$ is connected.   The theorem clearly holds for $n\le3$, so suppose $n\ge4$.
Let $p,q$ be the number of positive and negative eigenvalues of $G$,  respectively. Let $n$ be even; the theorem  follows similarly for odd $n$.
  If $p=1$, then by \cite[Theorem 6.7]{cds}, $G$ is a complete multipartite graphs with $s$ parts, say.
 If $s\le \frac{n}{2}+1$, $\he(G)=E(G)$ and we are done, so let $s\ge \frac{n}{2}+2$. Note that the complete graph $K_s$ is an induced subgraph of $G$, so by interlacing, $\la_{n-s+2}\le-1$. Therefore $\la_{\frac{n}{2}}\le-1$ and thus $\he(G)\ge n$ and this is greater than
 $2\sqrt{n-1}$ for $n\ge3$. So we may assume that $p\ge2$. This implies that  $q\ge2$ as well.
Note that either $p\le \frac{n}{2}+2$ or $q\le \frac{n}{2}+2$. Suppose $q\le \frac{n}{2}+2$; the proof is similar for the other case.  If we also have  $p\le \frac{n}{2}+2$, we are done. So let $p\ge \frac{n}{2}+2$.
  We observe that
\begin{align*}
    \e(G)&=2(\la_1+\cdots+\la_p)\\
&\le2(\la_1+\cdots+\la_r)+2(\la_{p-2r}+\cdots+\la_r)\\
&\le\he(G)+\frac{p-r}{r}\he(G).
\end{align*}
It turns out that
$$\he(G)\ge\frac{n}{2p}\e(G).$$
On the other hand, we see that $\e(G)\ge p+2$, as the energy of any graph is at least the rank of its adjacency matrix (\cite{Fa}, see also \cite{agz}). Combining the above inequalities we find
\begin{align*}
   \he(G)&\ge\frac{n}{2p}(p+2)\\
     &= \frac{n}{2}+\frac{n}{p}\\
     &\ge \frac{n^2}{2(n-2)},
\end{align*}
and the last value is greater  than $2\sqrt{n-1}$ for $n\ge4$.}
\end{proof}

\section{ A construction of ${\rm srg}(4t^2 +4t +2,\, 2t^2 +3t +1,\,
t^2 +2t,\, t^2 + 2t +1)$}

In this section we give an infinite family of strongly regular graphs with parameters $(n, k, \lambda, \mu) =
(4t^2 +4t +2, 2t^2 +3t +1,
t^2 +2t, t^2 + 2t +1)$ 
whose construction was communicated by Willem Haemers to us. He attributes the construction to J.J. Seidel.

Let $G$ be a graph with  vertex set $X$, and  $Y \subseteq X$.
From $G$ we obtain a new graph by leaving adjacency and non-adjacency inside $Y$ and
$X \setminus Y$ as it was, and interchanging adjacency and non-adjacency between $Y$ and
$X\setminus Y$ . This new graph
is said to be obtained by {\em Seidel
switching} with respect to the set of vertices $Y$.

Let $q=2t+1$ be a prime power.
Let $\Gamma$ be the Paley graph of order $q^2$, that is, the graph with vertex set GF$(q^2)$ and $x \sim y$ if $x-y$ is not a square in GF$(q^2)$.
Let $x$ be a primitive element of GF$(q^2)$ and consider $V= \{x^{i(q+1)} \mid i=0, \ldots, q-1\} \cup \{0\}$. Then $V$ is the subfield GF$(q)$ of GF$(q^2)$ and $V$ forms a coclique of size $q$.
Now $\{a + V \mid a \in \hbox{GF}(q^2)\}$ forms a partition into $q$
cocliques of size $q$.
Add an isolated vertex
and apply the Seidel switching with respect to the union of $t$ disjoint cocliques.
The resulting graph is a strongly regular graph with parameters  $(4t^2 +4t +2,\, 2t^2 +t,\,
t^2-1, \, t^2)$. Taking the complement of this graph give us a strongly regular graph with the required parameters. Note that for $t=1$, there exists only one such a graph namely the Johnson graph $J(5,2)$, for $t=2$ there are 10. For $t \leq 9$ they do exist and  $t=10$ seems the smallest open case.

\noindent{\bf Acknowledgements.} We are very grateful for the discussions with Patrick Fowler and Ivan Gutman.
Partick Fowler suggested the name of H\"uckel energy and provided us with the reference \cite{Fow}. Ivan Gutman presented us the history of the H\"uckel energy and provided us with the references \cite{c, cl, h31, h40}.

\end{document}